\documentclass[10pt,a4paper]{amsart}
\pdfoutput=1
\usepackage[sec,skipmbb]{rune}
\usepackage{runediag}

\DeclareMathOperator{\Tw}{Tw}
\newcommand{\TwL}{\Tw^{\ell}}

\newcommand{\LV}{\mathrm{LV}}
\newcommand{\IV}{\mathrm{IV}}

\newcommand{\DC}{\simp_{/\mathcal{C}}}
\newcommand{\Dst}{\simp_{*}}
\newcommand{\DCst}{\simp_{/\mathcal{C},*}}
\newcommand{\DCop}{\simp_{/\mathcal{C}}^{\op}}

\newcommand{\Yo}{\mathsf{y}}

\let\lim\relax \DeclareMathOperator{\lim}{lim}
\let\colim\relax \DeclareMathOperator{\colim}{colim}

\title{On (co)ends in $\infty$-categories}
\author{Rune Haugseng}
\address{Department of Mathematical Sciences, NTNU, Trondheim, Norway}
\urladdr{http://folk.ntnu.no/runegha}
\date{\today}

\begin{document}

\begin{abstract}
  In this short note we prove that two definitions of (co)ends in
  $\infty$-categories, via twisted arrow $\infty$-categories and via
  $\infty$-categories of simplices, are equivalent. We also show that
  weighted (co)limits, which can be defined as certain (co)ends, can
  alternatively be described as (co)limits over left and right
  fibrations, respectively.
\end{abstract}

\maketitle

\tableofcontents

\section{Introduction}
Ends and coends were first introduced by Yoneda~\cite{YonedaExt} and
play an important role in the theory of both ordinary and enriched
categories. Indeed, the calculus of (co)ends can be
viewed as a central organizing principle in category theory; we refer
the reader to the book \cite{LoregianEnd} for further discussion and
many applications of (co)ends. 

Ends and coends can also be defined in the context of \icats{}, though
it is not immediately clear how to do this if we use what seems to be
the most common way of defining (co)ends, in terms of so-called
``extranatural transformations''. Luckily, it is not actually
necessary to introduce this notion, as there are two easy ways to
define the end of a functor
$F \colon \mathcal{C}^{\op} \times \mathcal{C} \to \mathcal{D}$ as an
ordinary limit:
\begin{enumerate}[(A)]
\item Let $\TwL(\mathcal{C})$ denote the (left) twisted arrow
category of $\mathcal{C}$. This has morphisms in $\mathcal{C}$ as
objects, with a morphism from $f \colon x \to y$ to $f' \colon x' \to
y'$ given by a commutative diagram
\[
  \begin{tikzcd}
    x \arrow{d}[swap]{f}  & x'  \arrow{d}{f'} \arrow{l} \\
    y  \arrow{r} & y'. 
  \end{tikzcd}
\]
Taking the source and target of morphisms gives a functor
\[p \colon \TwL(\mathcal{C}) \to \mathcal{C}^{\op} \times \mathcal{C},\] and
the \emph{end} of $F$ is just the limit of the composite functor
\[F \circ p \colon \TwL(\mathcal{C}) \to \mathcal{D}.\]
\item If $\mathcal{D}$ has products, the end of $F$ can be
  expressed as the (reflexive) equalizer of the two morphisms
  \begin{equation}
    \label{eq:endequal}
   \prod_{x
      \in \mathcal{C}} F(x,x) 
    \rightrightarrows
    \prod_{f \colon x \to y} F(x,y),    
  \end{equation}
  given on the factor indexed by $f$ by projecting to $F(y,y)$ and
  $F(x,x)$ and composing with $f$ in the first and second variable,
  respectively. We can
  reinterpret this (and get rid of the assumption on $\mathcal{D}$)
  using the \emph{category of simplices} $\simp_{/\mathcal{C}}$ of
  $\mathcal{C}$. An object here is a functor $[n] \to \mathcal{C}$ with
  $[n]$ in $\simp$ (\ie{} a sequence of $n$ composable morphisms in
  $\mathcal{C}$ for $n \geq 0$), and a morphism is a
  commutative triangle
  \[ \opctriangle{{[n]}}{{[m]}}{\mathcal{C}}{\phi}{F}{G} \]
  for some morphism $\phi \colon [n] \to [m]$ in $\simp$. We then have
  a functor \[q \colon \simp_{/\mathcal{C}} \to \mathcal{C}^{\op}
  \times \mathcal{C}\] that takes $F \colon [n] \to \mathcal{C}$ to
  $(F(0), F(n))$ and a morphism as above to
  \[ (G(0) \to G(\phi(0)) = F(0), F(n) = G(\phi(n)) \to G(m)).\]
  The end of $F$ can then be defined as the limit of the composite
  \[ F \circ q \colon \simp_{/\mathcal{C}} \to \mathcal{D}.\]
  We can compute this in two stages by first taking the right Kan
  extension along the projection $\simp_{/\mathcal{C}} \to
  \simp$ and then taking the limit of the resulting cosimplicial
  diagram; for ordinary categories the inclusion of the subcategory of
  $\simp$ containing just the two coface maps $[0] \rightrightarrows [1]$
  is coinitial\footnote{We say a functor $F$ of \icats{} is \emph{cofinal}
    if the map on colimits induced by $F$ is always an equivalence,
    and \emph{coinitial} if $F^{\op}$ is cofinal, \ie{} the induced
    map on limits is always an equivalence.},   
  and this recovers the equalizer \cref{eq:endequal} we started
  with.\footnote{Pulling back this coinitial map along the right fibration
  $\simp_{/\mathcal{C}} \to \simp$, we see that 
  it is enough to consider the
  subcategory of $\simp_{/\mathcal{C}}$ consisting of functors $[n]
  \to \mathcal{C}$ with $n = 0,1$ and morphisms that lie over the two
  coface maps. This recovers the description of
  an end as a limit discussed in \cite{MacLaneWorking}*{\S IX.5}.}
\end{enumerate}
Both of these definitions have natural extensions to \icats{},
and our main
goal in this paper is to show that these definitions of \icatl{}
ends are in fact equivalent:
\begin{thm}
  For a functor of \icats{} $F \colon \mathcal{C}^{\op} \times
  \mathcal{C} \to \mathcal{D}$, the limits of the composites
  \[ \TwL(\mathcal{C}) \to \mathcal{C}^{\op} \times \mathcal{C} \to
    \mathcal{D}, \qquad \simp_{/\mathcal{C}} \to \mathcal{C}^{\op} \times
    \mathcal{C} \to \mathcal{D}\]
  are (naturally) equivalent if either exists.
\end{thm}
We introduce the \icat{} of simplices $\simp_{/\mathcal{C}}$ and
define the functor to $\mathcal{C}^{\op} \times \mathcal{C}$ in
\S\ref{sec:simpcat}. The definition of (co)ends in terms of twisted
arrow \icats{} was previously discussed in \cite{GlasmanTHHHodge}*{\S 2}
and \cite{freepres}; we review the definition in \S\ref{sec:tw} and
then prove the comparison.

For functors of \icats{} $W \colon \mathcal{C} \to \mathcal{S}$
and $\phi \colon \mathcal{C} \to \mathcal{D}$, we can define the
limit $\lim_{\mathcal{C}}^{W}\phi$ of $\phi$ \emph{weighted} by $W$ as the end of the functor
\[ \phi^{W} \colon \mathcal{C}^{\op} \times \mathcal{C} \to
  \mathcal{D},\] where $\phi(c)^{W(c')}$ denotes the limit over
$W(c')$ of the constant diagram with value $\phi(c)$, provided
$\mathcal{D}$ admits such limits. In \S\ref{sec:weight} we will prove
an alternative description of such weighted limits:
\begin{thm}
  Let $p \colon \mathcal{W} \to \mathcal{C}$ be the left fibration
  corresponding to the functor $W$. Then there is an equivalence
  \[ \lim^{W}_{\mathcal{C}} \phi \simeq \lim_{\mathcal{W}} \phi
    \circ p,\]
  provided either limit exists in $\mathcal{D}$.
\end{thm}
As a consequence of this result, the definition of weighted limits in
terms of ends agrees with that introduced by
Rovelli~\cite{RovelliWeight}. Moreover, this identifies the (co)end of
a functor
$F \colon \mathcal{C}^{\op} \times \mathcal{C} \to \mathcal{D}$ as the
(co)limit of $F$ weighted by the mapping space functor
$\Map_{\mathcal{C}}(\blank,\blank)$. We also identify the weighted
colimit functor
$\colim^{W}_{\mathcal{C}^{\op}} \colon \mathcal{P}(\mathcal{C}) \to
\mathcal{S}$ as the unique colimit-preserving functor extending
$W \colon \mathcal{C} \to \mathcal{S}$, while for $\mathcal{D}$ a
presentable \icat{} the weighted colimit functor can be identfied with
the presentable tensor product of this with $\id_{\mathcal{C}}$. This
characterization also specializes to give a further description of
coends.

\subsection*{Acknowledgments}
I thank Clark Barwick for introducing me to twisted arrow categories
and Grigory Kondyrev for decreasing my ignorance about weighted
colimits. I also thank Espen Auseth Nielsen for helpful discussions about
coends back in 2017, Moritz Groth for telling me the derivation of
the Bousfield--Kan formula for homotopy colimits given in \cref{BKform} at
some even earlier date, and Damien Lejay for pointing out the
characterization of the coend functor proved in \cref{cor:coendpsh}.

\section{(Co)ends via the $\infty$-Category of Simplices}\label{sec:simpcat}
In this section we define the \icat{} of simplices
$\simp_{/\mathcal{C}}$ of an \icat{} $\mathcal{C}$, and prove that
this has a canonical functor to $\mathcal{C}^{\op} \times
\mathcal{C}$, which allows us to give our first definition of
(co)ends.

\begin{defn}
  If $\mathcal{C}$ is an \icat{}, its \emph{$\infty$-category of
    simplices} $\DC$ is defined by the pullback square
  \dnlcsquare{\DC}{\Cat_{\infty/\mathcal{C}}}{\simp}{\CatI,}
  where the lower horizontal map is the usual embedding of $\simp$ in
  $\CatI$ (taking the ordered set $[n]$ to the corresponding
  category). Since $\Cat_{\infty/\mathcal{C}} \to \CatI$ is a right
  fibration, so is the projection $\DC \to \simp$.
\end{defn}

\begin{remark}
  The functor $\Dop \to \mathcal{S}$ corresponding to the right
  fibration $\DC \to \simp$ is given by
  \[ [n] \mapsto \Map_{\CatI}([n], \mathcal{C}).\]
  Thus this simplicial space is the \icat{} $\mathcal{C}$ viewed as a
  (complete) Segal space.
\end{remark}

\begin{warning}
  If $\mathcal{C}$ is an ordinary category, then
  $\simp_{/\mathcal{C}}$ as we have defined it here is not quite the
  category of simplices we discussed in the introduction, but a variant
  where the morphisms are triangles that commute up to a specified
  natural isomorphism. In other words, for us $\simp_{/\mathcal{C}}
  \to \simp$ is the fibration corresponding to the functor that takes
  $[n]$ to the \emph{groupoid} of functors $[n] \to \mathcal{C}$, rather
  than the \emph{set} of such functors.
\end{warning}

\begin{propn}\label{propn:pshrfibcolim}
  Suppose $\mathcal{I}$ is a small \icat{} and $\phi \colon
  \mathcal{I}^{\op} \to \mathcal{S}$ is the presheaf corresponding to
  a right fibration $p \colon \mathcal{E} \to \mathcal{I}$. Then
  $\phi$ is the colimit of the composite functor
  \[ \mathcal{E} \xto{p} \mathcal{I} \xto{\Yo} \Fun(\mathcal{I}^{\op},
    \mathcal{S}),\]
  where $\Yo$ is the Yoneda embedding.
\end{propn}
\begin{proof}
This is essentially \cite{HTT}*{Lemma 5.1.5.3}, but we include a proof.
Suppose $\psi$ is another presheaf on $\mathcal{I}$, corresponding to
a right fibration $\mathcal{F} \to \mathcal{I}$. Then unstraightening
gives a natural equivalence
 \[ \Map_{\Fun(\mathcal{I}^{\op}, \mathcal{S})}(\phi, \psi)
   \simeq
   \Map_{/\mathcal{I}}(\mathcal{E}, \mathcal{F}) \simeq
   \Map_{/\mathcal{E}}(\mathcal{E}, \mathcal{E} \times_{\mathcal{I}}
   \mathcal{F}).\]
 Here $\mathcal{E} \times_{\mathcal{I}}
 \mathcal{F} \to \mathcal{E}$ is the right fibration for the
 composite functor $\mathcal{E}^{\op} \xto{p^{\op}}
 \mathcal{I}^{\op} \xto{\psi} \mathcal{S}$, and we can identify 
 its \igpd{} of sections with the limit of this functor by \cite{HTT}*{Corollary
   3.3.3.4}. Thus we have a natural equivalence
  \[ \Map_{\Fun(\mathcal{I}^{\op}, \mathcal{S})}(\phi, \psi)
   \simeq \lim_{\mathcal{I}^{\op}} \psi \circ p^{\op} \simeq 
\lim_{\mathcal{I}^{\op}} \Map_{\Fun(\mathcal{I}^{\op},
  \mathcal{S})}(\Yo \circ p, \psi),\]
where the second equivalence follows from the Yoneda lemma. This shows
that $\phi$ has the universal property of the colimit
$\colim_{\mathcal{I}} \Yo \circ p$, as required.
\end{proof}

\begin{cor}
  The \icat{} $\mathcal{C}$ is the colimit of the composite functor
  \[ \DC \to \simp \to \CatI.\]
\end{cor}
\begin{proof}
  By thinking of \icats{} as complete Segal spaces, we can view
  $\CatI$ as a full subcategory of $\Fun(\Dop, \mathcal{S})$, and the
  composite $\Yo \colon \simp \to \CatI \hookrightarrow \Fun(\Dop, \mathcal{S})$ is the
  Yoneda embedding. Since colimits in $\CatI$ can be computed by
  taking colimits in $\Fun(\Dop, \mathcal{S})$ and then localizing, it
  is enough to show that $\mathcal{C}$ is the colimit of the composite
  \[ \DC \to \simp \xto{\Yo} \Fun(\Dop, \mathcal{S}),\]
which follows from \cref{propn:pshrfibcolim}.
\end{proof}

\begin{defn}
  Let $\simp_{*}$ denote the category with objects pairs $([n], i)$
  with $[n] \in \simp$ and $i \in [n]$, and a morphism $([n], i) \to
  ([m], j)$ given by a morphism $\phi \colon [n] \to [m]$ in $\simp$
  such that $\phi(i) \leq j$. Let $\pi \colon \simp_{*} \to
  \simp$ be the obvious projection.
\end{defn}

The functor $\pi$ is the cocartesian fibration for the inclusion
$\simp \hookrightarrow \CatI$; the cocartesian morphisms are the
morphisms over $\phi \colon [n] \to [m]$ of the form
$([n], i) \to ([m], \phi(i))$.  Thus the cocartesian fibration for the
composite $\DC \to \simp \to \CatI$ is
\[\pi_{\mathcal{C}} \colon \DCst := \DC \times_{\simp} \simp_{*} \to
  \DC.\]

\begin{cor}\label{cor:DCstloceq}
  There is a natural equivalence of \icats{}
  \[ \DCst[\txt{cocart}^{-1}] \isoto \mathcal{C}.\]
\end{cor}
\begin{proof}
  This follows from the description of colimits in $\CatI$ in
  \cite{HTT}*{\S 3.3.4}: the colimit of a functor $F \colon
  \mathcal{I} \to \CatI$ is given by inverting the cocartesian
  morphisms in the corresponding cocartesian fibration.
\end{proof}

\begin{defn}
  Let $l \colon \simp \to \simp_{*}$ be the section of $\pi$
  defined on objects by $l([n]) = ([n], n)$; since any map $\phi
  \colon [n] \to [m]$ satisfies $\phi(n) \leq m$ this makes
  sense. Note that $\pi l = \id$. Moreover, we have a natural isomorphism
  \[ \Hom_{\simp_{*}}(([n], i), l([m])) \cong \Hom_{\simp}([n],
  [m]),\]
  so that $l$ is right adjoint to $\pi$. Next, we define $\lambda
  \colon \simp_{*} \to \simp$ by $\lambda([n], i) = \{0,\ldots, i\}$;
  if $\phi \colon [n] \to [m]$ satisfies $\phi(i) \leq j$ then $\phi$
  restricts to a map $\{0,\ldots,i\} \to \{0,\ldots,j\}$, which we
  define to be $\lambda(\phi)$. A map $([n], n) \to ([m], i)$ in
  $\simp_{*}$ is determined by a map $[n] \to \lambda([m], i)$ in
  $\simp$, i.e.~we have a natural isomorphism
  \[ \Hom_{\simp_{*}}(l([n]), ([m], i)) \cong \Hom_{\simp}([n],
  \lambda([m], i)),\]
  and so $\lambda$ is right adjoint to $l$. Note also that $\lambda l
  = \id$ and there is a natural transformation $\alpha \colon \lambda
  \to \pi$ given at the object $([n], i)$ by the inclusion
  $\{0,\ldots,i\} \hookrightarrow [n]$. (This can also be defined as
  $\lambda$ applied to the unit transformation $\id \to l \pi$.)
\end{defn}

\begin{lemma}\label{lem:LVloceq}
  Let $\LV$ denote the set of \emph{last-vertex morphisms} in $\simp$,
  \ie{} the maps $\phi \colon [n] \to [m]$
   such that $\phi(n) = m$. Then
  \begin{enumerate}[(i)]
  \item $\lambda$ takes the $\pi$-cocartesian morphisms to morphisms
    in $\LV$,
  \item $l$ takes morphisms in $\LV$ to $\pi$-cocartesian
    morphisms,
  \item the unit map $[n] \to \lambda l([n])$ is in $\LV$ (being
    in fact the identity of $[n]$),
  \item the counit map $l\lambda([n],i) = ([i],i) \to ([n], i)$ is
    $\pi$-cocartesian.
  \end{enumerate}
  Moreover, the adjunction $l \dashv \lambda$ induces an adjoint
  equivalence
  \[ \simp[\LV^{-1}] \mathop{\rightleftarrows}^{\sim}_{\sim} \simp_{*}[\txt{cocart}^{-1}].\]
\end{lemma}
\begin{proof}
  Properties (i)--(iv) are immediate from the definition of the
  $\pi$-cocartesian morphisms, and imply that the adjunction $l \dashv
  \lambda$ descends to the localized \icats{} where the unit and
  counit transformations become natural equivalences.
\end{proof}

We now want to lift this equivalence to $\DC$; we first lift the
adjoint triple:
\begin{propn}
  The adjoint triple $\pi \dashv l \dashv \lambda$ induces for all
  \icats{} $\mathcal{C}$ an adjoint triple of functors
  $\pi_{\mathcal{C}} \dashv l_{\mathcal{C}} \dashv
  \lambda_{\mathcal{C}}$ between $\DCst$ and $\DC$.
\end{propn}
\begin{proof}
  Since $\pi l = \id$, pulling back $l$ gives a commutative diagram
  \[
    \begin{tikzcd}
      \DC \arrow{r}{l_{\mathcal{C}}} \arrow{d} \arrow[bend left,equals]{rr} & \DCst
      \arrow{r}{\pi_{\mathcal{C}}} \arrow{d} & \DC \arrow{d} \\
      \simp \arrow{r}{l} \arrow[bend right,equals]{rr} & \Dst \arrow{r}{\pi} & \simp,
    \end{tikzcd}
  \]
  where both squares are cartesian. The unit transformation $\id \to l
  \pi$ pulls back similarly, and the adjunction identities hold since
  they lie over equivalences in $\simp$ and $\Dst$ and right
  fibrations are conservative.

  We define $\lambda_{\mathcal{C}} \colon \DCst \to \DC$ by taking the
  cartesian pullback of $\alpha \colon \lambda \to \pi$, which gives a
  filler in the diagram
  \[
    \begin{tikzcd}
      \DCst \times \{1\} \arrow{r}{\pi_{\mathcal{C}}} \arrow[hookrightarrow]{d} & \DC
      \arrow{dd} \\
      \DCst \times \Delta^{1} \arrow{d}
      \arrow[dashed]{ur}{\alpha_{\mathcal{C}}} \\
      \Dst \times \Delta^{1} \arrow{r}{\alpha} & \simp
    \end{tikzcd}
  \]
  where the value of $\alpha_{\mathcal{C}}$ at an object $X \in \DCst$
  over $([n],i)$ in $\Dst$ is the
  cartesian morphism in $\DC$ over $\alpha_{([n],i)}$ with target
  $\pi_{\mathcal{C}}X$. Since $\alpha$ restricts to the identity
  transformation along $l$, it follows that $\lambda_{\mathcal{C}}
  \circ l_{\mathcal{C}} \simeq \id$.

  The natural transformation $l\alpha \colon l\lambda \to l \pi$ factors
  as the composite $l \lambda \to \id \to l \pi$ of the counit
  transformation for $l \dashv \lambda$ and the unit transformation for
  $\pi \dashv l$. Hence $l_{\mathcal{C}}\alpha_{\mathcal{C}}$ factors
  as $l_{\mathcal{C}} \lambda_{\mathcal{C}} \to \id \to
  l_{\mathcal{C}}\pi_{\mathcal{C}}$ where the second morphism is the
  unit for $\pi_{\mathcal{C}} \dashv l_{\mathcal{C}}$, since this is the
  unique transformation over $\id \to l\pi$ with target
  $l_{\mathcal{C}}\pi_{\mathcal{C}}$, as $\DCst \to \Dst$ is a right
  fibration. We claim that the transformation
  $l_{\mathcal{C}}\lambda_{\mathcal{C}} \to \id$ is a counit. To see
  this consider the commutative diagrams
  \[
    \begin{tikzcd}
      \Map_{\DC}(X, \lambda_{\mathcal{C}}Y) \arrow{r} \arrow{d} &
      \Map_{\DCst}(l_{\mathcal{C}}X,
      l_{\mathcal{C}}\lambda_{\mathcal{C}}Y) \arrow{d} \arrow{r} &
      \Map_{\DCst}(l_{\mathcal{C}}X, Y) \arrow{d} \\
      \Map_{\simp}([m], \lambda([n],i)) \arrow{r} &
      \Map_{\Dst}(l[m], l\lambda([n],i)) \arrow{r} & \Map_{\Dst}(l[m],
      ([n],i)). 
    \end{tikzcd}
  \]
for objects $X$ and $Y$ lying over $[m]$ and
  $([n], i)$, respectively. Here the left square is cartesian since
  $l_{\mathcal{C}}$ is a pullback of $l$, and the right square is
  cartesian since $\DCst \to \Dst$ is a right fibration (and hence the
  morphism $l_{\mathcal{C}}\lambda_{\mathcal{C}}Y \to Y$ is
  cartesian). The composite in the bottom row is an equivalence,
  as we know that $l$ is left adjoint to $\lambda$, so this implies
  that the composite in the top row is an equivalence, as required.
\end{proof}

\begin{cor}
  Let $\LV_{\mathcal{C}}$ denote the morphisms in $\DC$ that lie over
  last-vertex morphisms in $\simp$. Then
  \begin{enumerate}
  \item $\lambda_{\mathcal{C}}$ takes  $\pi_{\mathcal{C}}$-cocartesian
    morphisms to morphisms in $\LV_{\mathcal{C}}$,
  \item $l_{\mathcal{C}}$ takes morphisms in $\LV_{\mathcal{C}}$ to
    $\pi_{\mathcal{C}}$-cocartesian morphisms,
  \item the unit map $X \to \lambda_{\mathcal{C}}l_{\mathcal{C}}X$ is
    in $\LV_{\mathcal{C}}$ (since it is an equivalence),
  \item the counit map $l_{\mathcal{C}}\lambda_{\mathcal{C}}X \to X$
    is $\pi_{\mathcal{C}}$-cocartesian.
  \end{enumerate}
  Moreover, the adjunction $l_{\mathcal{C}} \dashv
  \lambda_{\mathcal{C}}$ induces an adjoint equivalence
  \[ \DC[\LV^{-1}] \mathop{\rightleftarrows}^{\sim}_{\sim}
  \DCst[\txt{cocart}^{-1}].\]
\end{cor}
\begin{proof}
  The $\pi_{\mathcal{C}}$-cocartesian morphisms are precisely the
  morphisms in $\DCst$ that lie over $\pi$-cocartesian morphisms in
  $\Dst$, so this follows from \cref{lem:LVloceq} and the fact that
  the unit and counit for $l_{\mathcal{C}} \dashv
  \lambda_{\mathcal{C}}$ lie over the unit and counit for $l \dashv \lambda$.
\end{proof}

Combining this with the equivalence of \cref{cor:DCstloceq}, we have proved:
\begin{propn}
  There is a natural equivalence of \icats{}
  \[ \DC[\LV^{-1}] \simeq \mathcal{C},\]
  and hence a natural transformation $\mathfrak{L}_{\mathcal{C}}
  \colon  \DC \to \mathcal{C}$.\qed
\end{propn}

To obtain a natural map $\DC \to \mathcal{C}^{\op}$, we combine this
with the order-reversing automorphism of $\simp$:
\begin{defn}
  Let $\txt{rev} \colon \simp \to \simp$ be the order-reversing automorphism of
  $\simp$, i.e.~$\txt{rev}([n]) = [n]$ but for $\phi \colon [n] \to
  [m]$ we have $\txt{rev}(\phi)(i) = m-\phi(n-i)$.
\end{defn}
\begin{lemma}
  We have a natural pullback square
  \[\csquare{\simp_{/\mathcal{C}^{\op}}}{\DC}{\simp}{\simp.}{\txt{rev}_{\mathcal{C}}}{}{}{\txt{rev}}\]
  Since $\txt{rev}$ is an equivalence, so is $\txt{rev}_{\mathcal{C}}$.
\end{lemma}
\begin{proof}
  The pullback of $\DC \to \simp$ along $\txt{rev}$ is the right
  fibration for the composite \[\Dop \xto{\txt{rev}^{\op}} \Dop
  \xto{\mathcal{C}} \mathcal{S}\]
  and this composite is precisely the complete Segal space
  corresponding to $\mathcal{C}^{\op}$.
\end{proof}

Under the equivalence $\txt{rev}$, the last-vertex
morphisms in $\LV$ correspond to the \emph{initial-vertex
  morphisms} $\IV$, \ie{} the maps $\phi \colon [n] \to [m]$ such
that $\phi(0) = 0$. We thus get:
\begin{cor}
  There is a natural equivalence of \icats{}
  \[ \DC[\IV_{\mathcal{C}}^{-1}] \simeq \mathcal{C}^{\op},\]
  where $\IV_{\mathcal{C}}$ are the morphisms in $\DC$ that lie over
  $\IV$. Hence there is a natural transformation
  $\mathfrak{I}_{\mathcal{C}} \colon \DC \to \mathcal{C}^{\op}$.
\end{cor}

\begin{propn}
  Suppose $L \colon \mathcal{C} \to \mathcal{C}[W^{-1}]$ is the localization of
  an \icat{} $\mathcal{C}$ at a collection of morphisms $W$. Then the
  functor $L$
  is coinitial and cofinal.
\end{propn}
\begin{proof}
  Without loss of generality the morphisms in $W$ are closed under
  composition and contain all equivalences; we can then let $\mathcal{W}$
  denote the subcategory of $\mathcal{C}$ containing the morphisms in
  $W$. By definition of $\mathcal{C}[W^{-1}]$ we then have a pushout
  square
  \dnlcsquare{\mathcal{W}}{\|\mathcal{W}\|}{\mathcal{C}}{\mathcal{C}[W^{-1}],}
  where $\|\mathcal{W}\|$ denotes the \igpd{} obtained by inverting
  all morphisms in $\mathcal{W}$. By \cite{HTT}*{Corollary 4.1.2.6}
  the map $\mathcal{W} \to \|\mathcal{W}\|$ is cofinal, so by
  \cite{HTT}*{Corollary 4.1.2.7} the pushout $\mathcal{C} \to \mathcal{C}[W^{-1}]$
  is also cofinal. To see that it is also coinitial, we apply the same
  argument on opposite \icats{}.
\end{proof}
\begin{cor}\label{cor:DCmapscof}
  For any \icat{} $\mathcal{C}$, the functors
  \[ \mathfrak{L}_{\mathcal{C}} \colon \DC \to \mathcal{C}, \quad
    \mathfrak{I}_{\mathcal{C}} \colon \DC \to \mathcal{C}^{\op}\]
  are both coinitial and cofinal.\qed
\end{cor}
That $\mathfrak{I}_{\mathcal{C}}$ is coinitial was previously shown by
Shah~\cite{ShahThesis}*{Proposition 12.4}. As a consequence, we obtain
an \icatl{} version of the Bousfield--Kan
formula for homotopy colimits, which has also previously appeared in
work of Shah~\cite{ShahThesis}*{Corollary 12.5} and
Mazel-Gee~\cite{MazelGeeGroth}*{Theorem 5.8}.
\begin{cor}[Bousfield--Kan formula]\label{BKform}
  Let $\mathcal{D}$ be a cocomplete \icat{}. The colimit of a functor
  $F \colon \mathcal{C} \to \mathcal{D}$ is equivalent to the colimit
  of a simplicial object $\Dop \to \mathcal{D}$ given by
  \[ [n] \mapsto \colim_{\alpha \in \Map([n], \mathcal{C})} F(\alpha(0)).\]
\end{cor}
\begin{proof}
  We can compute the colimit of $F$ after composing with the cofinal
  map $\mathfrak{I}_{\mathcal{C}}^{\op} \colon \DCop \to \mathcal{C}$, which takes $\alpha \colon [n] \to
  \mathcal{C}$ to $\alpha(0)$. This colimit we can in turn compute in
  two stages, by first taking the left Kan extension along the
  projection $\DCop \to \Dop$, which produces a simplicial object of
  the given form, and then taking the colimit of this simplicial object.
\end{proof}

We are now in a position to define ends and coends:
\begin{defn}
  Given a functor
  $F \colon \mathcal{C} \times \mathcal{C}^{\op} \to \mathcal{D}$, its
  \emph{coend}\footnote{We use the original notational convention of
    \cite{YonedaExt} rather than the ``Australian'' convention, where
    the coend is denoted $\int^{\mathcal{C}}F$ and the end is denoted
    $\int_{\mathcal{C}}F$ --- after all, it is the \emph{co}end of $F$
    that
    is somewhat analogous to an integral, not the end.}
  $\int_{\mathcal{C}} F$ is the colimit of the composite
  functor
  \[ \DCop
    \xto{(\mathfrak{I}_{\mathcal{C}}^{\op},\mathfrak{L}_{\mathcal{C}}^{\op})}
    \mathcal{C} \times \mathcal{C}^{\op}\xto{F} \mathcal{D}.\] 
  Dually, the
  \emph{end} $\int^{*}_{\mathcal{C}}F$ of $F$ is the limit of the
  composite functor
  \[ \DC \xto{(\mathfrak{L}_{\mathcal{C}}, \mathfrak{I}_{\mathcal{C}})} \mathcal{C} \times \mathcal{C}^{\op} \to \mathcal{D}.\]
\end{defn}

\begin{lemma}\label{lem:coendsimpform}
  If $\mathcal{D}$ is a cocomplete \icat{}, then the coend of a
  functor \[F \colon \mathcal{C} \times \mathcal{C}^{\op} \to \mathcal{D}\] can
  be computed as the colimit of a simplicial object $\Dop \to
  \mathcal{D}$ given by 
  \[ [n] \mapsto \colim_{\alpha \in \Map([n], \mathcal{C})}
  F(\alpha(0), \alpha(n)).\]
\end{lemma}
\begin{proof}
  The colimit over $\DCop$ can be computed in two steps by first taking the left
  Kan extension along the projection $\DCop \to \Dop$, which gives the
  desired simplicial object, and then taking the colimit of this
  simplicial object.
\end{proof}

A key property of ends is the ``Fubini theorem'' for iterated
ends. This was proved for \icats{} by Loregian~\cite{LoregianFub},
using the definition of (co)ends via twisted arrows. We include a 
proof, as it is very easy to see using \icats{} of simplices:
\begin{propn}[``Fubini's Theorem'']
  Given a functor $F \colon (\mathcal{C}\times \mathcal{D})^{\op}
  \times (\mathcal{C} \times \mathcal{D}) \to \mathcal{E}$, there are 
  natural equivalences of ends 
  \[ \int_{\mathcal{C}}^{*}\int_{\mathcal{D}}^{*} F \simeq
    \int_{\mathcal{C} \times \mathcal{D}}^{*} F \simeq \int_{\mathcal{D}}^{*}\int_{\mathcal{C}}^{*} F.\]
\end{propn}
\begin{proof}
  Since unstraightening preserves limits, we have a natural
  equivalence \[\simp_{/\mathcal{C} \times \mathcal{D}} \simeq \DC
  \times_{\simp} \simp_{/\mathcal{D}}.\]
  This means we have a pullback square
  \[
    \csquare{\simp_{/\mathcal{C} \times \mathcal{D}}}{\DC \times
      \simp_{/\mathcal{D}}}{\simp}{\simp \times \simp,}{}{}{}{}
  \]
  where the bottom horizontal arrow is coinitial, since $\Dop$ is
  sifted. Since the right vertical arrow is a right fibration, this
  implies that the top horizontal arrow is also coinitial. Moreover,
  the composite of this functor with the projection to $\DC$ is the
  functor induced by the composition with the projection $\mathcal{C}
  \times \mathcal{D} \to \mathcal{C}$, and similarly for
  $\simp_{/\mathcal{D}}$. It follows that we also have a commutative
  triangle
  \[
    \begin{tikzcd}
      \simp_{/\mathcal{C} \times \mathcal{D}} \arrow{rr}
      \arrow{dr}[swap]{(\mathfrak{I}_{\mathcal{C} \times \mathcal{D}},
        \mathfrak{L}_{\mathcal{C} \times \mathcal{D}})} & &
      \simp_{/\mathcal{C}} \times \simp_{/\mathcal{D}}
      \arrow{dl}{(\mathfrak{I}_{\mathcal{C}},
        \mathfrak{I}_{\mathcal{D}}, \mathfrak{L}_{\mathcal{C}},
        \mathfrak{L}_{\mathcal{D}})} \\
       & \mathcal{C}^{\op} \times \mathcal{D}^{\op} \times \mathcal{C}
       \times \mathcal{D}.
     \end{tikzcd}
   \]
   Together with the description of limits over a product as iterated
   limits this implies the result.
\end{proof}

\section{Coends via the Twisted Arrow $\infty$-Category}\label{sec:tw}
In this section we recall the definition of twisted arrow \icats{},
and prove that we can equivalently define (co)ends as (co)limits using
these \icats{}.

\begin{defn}
  Let $\epsilon \colon \simp \to \simp$ be the endomorphism given by
  \[ [n] \mapsto [n]^{\op} \star [n], \]
  and write $\iota \colon \id \to \epsilon$, $\rho \colon \txt{rev} \to
  \epsilon$ for the natural transformations corresponding to the
  inclusions of the factors $[n]$ and $[n]^{\op}$.
  For $\mathcal{C}$ an \icat{}, we define $\TwL(\mathcal{C})$ as the
  simplicial space
  \[ [n] \mapsto \Map([n]^{\op} \star [n], \mathcal{C}),\]
  \ie{} $\epsilon^{*}\mathcal{C}$ if we view $\mathcal{C}$ as a
  complete Segal space.  Restricting along $\iota$ and $\rho$  we get a projection
  \[ \eta_{\mathcal{C}} \colon \TwL(\mathcal{C}) \to \mathcal{C}^{\op} \times \mathcal{C}.\]
  We refer to $\TwL(\mathcal{C})$ as the (left) \emph{twisted arrow
    \icat{}} of $\mathcal{C}$, as is justified by the following result:
\end{defn}

\begin{propn}[\cite{cois}*{Proposition A.2.3}, \cite{HA}*{Proposition 5.2.1.3}]
  If $\mathcal{C}$ is an \icat{} then $\TwL(\mathcal{C})$ is a
  complete Segal space, and the projection $\eta_{\mathcal{C}} \colon \TwL(\mathcal{C}) \to
  \mathcal{C}^{\op} \times \mathcal{C}$ is a left fibration. \qed
\end{propn}

\begin{variant}
  If we instead consider the endofunctor of $\simp$ given by $[n]
  \mapsto [n] \star [n]^{\op}$, we get the \emph{right} twisted arrow
  \icat{} $\Tw^{r}(\mathcal{C}) := \TwL(\mathcal{C})^{\op}$, whose
   projection $\Tw^{r}(\mathcal{C}) \to \mathcal{C} \times
  \mathcal{C}^{\op}$ is a right fibration.
\end{variant}

\begin{lemma}\label{lem:DTwCpb}
  There is a natural pullback square
  \[\csquare{\simp_{/\TwL(\mathcal{C})}}{\DC}{\simp}{\simp}{\epsilon_{\mathcal{C}}}{}{}{\epsilon}\]
\end{lemma}
\begin{proof}
  The pullback of $\DC \to \simp$ along $\epsilon$ is the right
  fibration for the composite
  \[ \simp \xto{\epsilon} \simp \xto{\mathcal{C}} \mathcal{S},\]
  which is by definition $\simp_{/\TwL(\mathcal{C})}$.
\end{proof}

\begin{propn}\label{propn:twsquare}
  There is a natural commutative square
  \begin{equation}
    \label{eq:twsquare}
  \csquare{\simp_{/\TwL(\mathcal{C})}}{\DC}{\TwL(\mathcal{C})}{\mathcal{C}^\op
    \times
    \mathcal{C}.}{\epsilon_{\mathcal{C}}}{\mathfrak{L}_{\TwL(\mathcal{C})}}{(\mathfrak{I}_{\mathcal{C}},\mathfrak{L}_{\mathcal{C}})}{\eta_{\mathcal{C}}}     
  \end{equation}
\end{propn}
\begin{proof}
  Observe that the definition of $\epsilon$ implies that we have $\epsilon(\LV)\subseteq \IV \cap
  \LV$. The composite
  \[\Dop_{/\TwL(\mathcal{C})} \xto{\epsilon_{\mathcal{C}}} \Dop_{/\mathcal{C}} \xto{(\mathfrak{I}_{\mathcal{C}},\mathfrak{L}_{\mathcal{C}})} \mathcal{C}^{\op}
  \times \mathcal{C}\] hence takes the morphisms in $\LV_{\TwL(\mathcal{C})}$ to
  equivalences, and so this composite factors uniquely through the localization
  \[\mathfrak{L}_{\TwL(\mathcal{C})} \colon \Dop_{/\TwL(\mathcal{C})}
  \to \Dop_{/\TwL(\mathcal{C})}[\LV^{-1}] \simeq
  \TwL(\mathcal{C}),\] giving a natural commutative square
  \[\csquare{\simp_{/\TwL(\mathcal{C})}}{\DC}{\TwL(\mathcal{C})}{\mathcal{C}^\op
      \times
      \mathcal{C}.}{\epsilon_{\mathcal{C}}}{\mathfrak{L}_{\TwL(\mathcal{C})}}{(\mathfrak{I}_{\mathcal{C}},\mathfrak{L}_{\mathcal{C}})}{\upsilon_{\mathcal{C}}} \]
  It remains to show that the induced functor $\upsilon_{\mathcal{C}}$
  is naturally equivalent to $\eta_{\mathcal{C}}$.

  Viewing the natural transformation $\iota$ as a functor $\simp
  \times \Delta^{1} \to \simp$, the projection $p \colon \TwL(\mathcal{C}) \to
  \mathcal{C}$ is described as a morphism of complete Segal spaces by
  $\iota^{*}\mathcal{C} \colon \Dop \times (\Delta^{1})^{\op} \to
  \mathcal{S}$, corresponding to the right fibration obtained as a
  pullback
  \begin{equation}
    \label{eq:iotasquare}
    \csquare{\iota^{*}\DC}{\DC}{\simp \times
      \Delta^{1}}{\simp}{\iota_{\mathcal{C}}}{}{}{\iota}
  \end{equation}
  The composite functor $\iota^{*}\DC \to \Delta^{1}$ is a cartesian
  fibration, and corresponds to the functor $q \colon \simp_{/\TwL(\mathcal{C})}
  \to \DC$ given by composition with $p$, which fits in a commutative
  square
  \[
    \csquare{\simp_{/\TwL(\mathcal{C})}}{\DC}{\TwL(\mathcal{C})}{\mathcal{C}}{q}{\mathfrak{L}_{\TwL(\mathcal{C})}}{\mathfrak{L}_{\mathcal{C}}}{p}
  \]
  It therefore suffices to show that $q$ is equivalent to
  $\epsilon_{\mathcal{C}}$; to see this we observe that the
  pullback square
  \cref{eq:iotasquare} induces a commutative triangle
  \[
    \begin{tikzcd}
      \iota^{*}\DC \arrow{rr} \arrow{dr} & & \DC \times
      \Delta^{1}\arrow{dl} \\
       & \Delta^{1},
    \end{tikzcd}
  \]
  where the diagonal functors are both cartesian fibrations and the
  horizontal functor preserves cartesian morphisms. We can straighten
  this to a commutative square of \icats{}
  \[
    \begin{tikzcd}
      \simp_{/\TwL(\mathcal{C})}
      \arrow{d}[swap]{\epsilon_{\mathcal{C}}} 
      \arrow{r}{q} & \DC \arrow[equals]{d} \\
      \DC \arrow[equals]{r} & \DC,
    \end{tikzcd}
  \]
  which implies that $q \simeq \epsilon_{\mathcal{C}}$.
  A similar argument works for the projection $\TwL(\mathcal{C}) \to
  \mathcal{C}^{\op}$, which completes the proof.
\end{proof}

The following is a special case of \cite{BarwickMackey}*{Proposition
  2.1}.
\begin{propn}\label{epscoinit}
  $\epsilon \colon \simp \to \simp$ is coinitial. 
\end{propn}
\begin{proof}
  By \cite{HTT}*{Theorem 4.1.3.1} it suffices to show that the
  pullback $\simp \times_{\simp} \simp_{/[n]}$ along $\epsilon$ is  weakly contractible
  for all objects $[n]$ in $\simp$. This pullback we can identify with
  $\simp_{/\TwL [n]}$
  by \cref{lem:DTwCpb}, and we have a cofinal functor
  $\mathfrak{L}_{\TwL [n]} \colon \simp_{/\TwL
    [n]} \to \TwL [n]$ from \cref{cor:DCmapscof}. Since cofinal
  functors are in particular weak homotopy equivalences, it suffices
  to show that the category $\TwL [n]$ is weakly contractible. This
  category can be described as the partially ordered set of pairs
  $(i,j)$ with $0 \leq i \leq j \leq n$, with partial ordering given
  by
  \[ (i,j) \leq (i',j') \iff i' \leq i \leq j \leq j'.\]
  Here $(0,n)$ is a terminal object, and so this category is indeed
  weakly contractible.
\end{proof}

\begin{cor}\label{cor:epsCcoinit}
  The functor $\epsilon_{\mathcal{C}} \colon \simp_{/\TwL(\mathcal{C})} \to \DC$ is coinitial.
\end{cor}
\begin{proof}
  From \cref{epscoinit} and \cref{lem:DTwCpb} we know that this
  functor is the pullback of the coinitial functor $\epsilon$ along
  the cartesian fibration $\DC \to \simp$. It is therefore coinitial
  by (the dual of) \cite{HTT}*{Proposition 4.1.2.15}.
\end{proof}

\begin{cor}\label{cor:endviatw}
  Given a functor $F \colon \mathcal{C}^{\op} \times \mathcal{C} \to
  \mathcal{D}$, its end is given by the limit of the composite
  \[ \TwL(\mathcal{C}) \xto{\eta_{\mathcal{C}}} \mathcal{C}^{\op} \times \mathcal{C} \xto{F} \mathcal{D}.\]
\end{cor}
\begin{proof}
  In the commutative square \cref{eq:twsquare} from
  \cref{propn:twsquare}, the functor $\epsilon_{\mathcal{C}}$ is
  coinitial by \cref{cor:epsCcoinit} while the functor
  $\mathfrak{L}_{\TwL(\mathcal{C})}$ is coinitial by
  \cref{cor:DCmapscof}. We therefore have natural equivalences
  \[
    \begin{split}
    \lim_{\TwL(\mathcal{C})} F \circ \eta_{\mathcal{C}} & \simeq  \lim_{\simp_{/\TwL(\mathcal{C})}} F \circ \eta_{\mathcal{C}} \circ
    \mathfrak{L}_{\TwL(\mathcal{C})} \\
     & \simeq \lim_{\simp_{/\TwL(\mathcal{C})}} F \circ
    (\mathfrak{I}_{\mathcal{C}},\mathfrak{L}_{\mathcal{C}}) \circ
    \epsilon_{\mathcal{C}} \\
    &      \simeq \lim_{\simp_{/\mathcal{C}}} F \circ (\mathfrak{I}_{\mathcal{C}},\mathfrak{L}_{\mathcal{C}}),
    \end{split}
  \]
  where the latter is the end of $F$ as we defined it above.
\end{proof}

\begin{remark}\label{rmk:coendtw}
  Dually, for a functor $F \colon \mathcal{C} \times \mathcal{C}^{\op}
  \to \mathcal{D}$, its coend can be computed either as the colimit of
  the composite 
  \[ \DCop \xto{(\mathfrak{I}_{\mathcal{C}}^{\op},
      \mathfrak{L}_{\mathcal{C}}^{\op})} \mathcal{C} \times
      \mathcal{C}^{\op} \xto{F} \mathcal{D},\]
    or as the colimit of
    \[ \Tw^{r}(\mathcal{C}) \simeq \TwL(\mathcal{C})^{\op}
      \xto{\eta_{\mathcal{C}}^{\op}} \mathcal{C} \times
      \mathcal{C}^{\op} \xto{F} \mathcal{D}.\]
\end{remark}

\section{Weighted (Co)limits}\label{sec:weight}
Weighted (co)limits can be defined as certain (co)ends. Our goal in this
section is to show that they can also be expressed as (co)limits over
left and right fibrations, respectively. The latter description agrees with the definition of weighted
(co)limits studied by Rovelli~\cite{RovelliWeight} in terms of an explicit construction in
quasicategories.

Given a presheaf $W \colon \mathcal{I}^{\op} \to \mathcal{S}$ and a
functor $\phi \colon \mathcal{I} \to \mathcal{C}$, the \emph{colimit
  of $\phi$ weighted by $W$}, denoted $\colim_{\mathcal{I}}^{W}\phi$,
can be defined as the coend of the functor
$W \times \phi \colon \mathcal{I}^{\op} \times \mathcal{I} \to
\mathcal{C}$, at least if $\mathcal{C}$ admits colimits indexed by
$\infty$-groupoids. Similarly, for
$\psi \colon \mathcal{I}^{\op} \to \mathcal{C}$ the \emph{limit of
  $\psi$ weighted by $W$}, denoted
$\lim^{W}_{\mathcal{I}^{\op}} \phi$, can be defined as the end of the
functor
$\psi^{W} \colon \mathcal{I} \times \mathcal{I}^{\op} \to
\mathcal{C}$, provided $\mathcal{C}$ admits limits indexed by
\igpds{}. One can also characterize weighted limits and colimits in
terms of universal properties, as we have
\[ \Map_{\mathcal{C}}(\colim^{W}_{\mathcal{I}} \phi, c) \simeq
  \lim^{W}_{\mathcal{I}^{\op}} \Map_{\mathcal{C}}(\phi, c), \]
\[ \Map_{\mathcal{C}}(c, \lim^{W}_{\mathcal{I}^{\op}} \psi) \simeq
  \lim^{W}_{\mathcal{I}^{\op}} \Map_{\mathcal{C}}(c,\psi). \] Since
all weighted limits exist in $\mathcal{S}$, this also gives a
definition of weighted (co)limits without any assumptions on
$\mathcal{C}$.

The key property of weighted limits is that they describe mapping
spaces in functor categories. We state this in the case of presheaves:
\begin{thm}[Glasman]\label{thm:natend}
  For presheaves $\phi, \psi \in \mathcal{P}(\mathcal{I})$ we have a
  natural equivalence
  \[ \Map_{\mathcal{P}(\mathcal{I})}(\phi, \psi) \simeq
    \lim_{\mathcal{I}^{\op}}^{\phi} \psi.\]
\end{thm}
\begin{proof}
  This is a special case of \cite{GlasmanTHHHodge}*{Proposition 2.3}
  or \cite{freepres}*{Proposition 5.1}.
\end{proof}

\begin{remark}\label{rmk:Yonedawt}
  As a consequence, the Yoneda lemma implies that we can express any
  presheaf as a colimit weighted by itself:
  \begin{equation}
    \label{eq:Yonedacolim}
    \phi \simeq \colim^{\phi}_{\mathcal{I}}
    \Yo_{\mathcal{I}}
  \end{equation}
  for $\phi \in \mathcal{P}(\mathcal{I})$,
  where $\Yo_{\mathcal{I}} \colon \mathcal{I} \to
  \mathcal{P}(\mathcal{I})$ is the Yoneda embedding.
  This follows from the equivalences
  \[\Map_{\mathcal{P}(\mathcal{I})}(\colim^{\phi}_{\mathcal{I}} \Yo_{\mathcal{I}}, \psi)
    \simeq \lim^{\phi}_{\mathcal{I}^{\op}}
    \Map_{\mathcal{P}(\mathcal{I})}(\Yo_{\mathcal{I}}, \psi) \simeq
    \lim^{\phi}_{\mathcal{I}^{\op}} \psi \simeq
    \Map_{\mathcal{P}(\mathcal{I})}(\phi, \psi).\]
\end{remark}

\begin{propn}\label{wtlimlfib}
  Suppose $q \colon \mathcal{V} \to \mathcal{J}$ is the left fibration
  corresponding to a functor $V \colon \mathcal{J} \to
  \mathcal{S}$. Then for a functor $\psi \colon \mathcal{J} \to
  \mathcal{C}$ there is an equivalence
  \begin{equation}
    \label{eq:wtlimlfib}
   \lim^{V}_{\mathcal{J}} \psi \simeq \lim_{\mathcal{V}} \psi \circ
    q,
  \end{equation}
  provided either side exists.
\end{propn}
\begin{proof}
  By the universal mapping properties of the two sides it suffices to
  show there is an equivalence
  \[ \lim^{V}_{\mathcal{J}} \Map_{\mathcal{C}}(c, \psi) \simeq
    \lim_{\mathcal{V}} \Map_{\mathcal{C}}(c, \psi \circ q),\]
  natural in $c \in \mathcal{C}$. In other words, it suffices to show
  there is a natural equivalence \cref{eq:wtlimlfib} for functors
  $\Psi \colon \mathcal{J} \to \mathcal{S}$.
  
  Using \cref{thm:natend} and the straightening equivalence, we can
  rewrite the left-hand side as
  \[ \lim^{V}_{\mathcal{J}} \Psi \simeq
    \Map_{\Fun(\mathcal{J},\mathcal{S})}(V, \Psi) \simeq
    \Map_{/\mathcal{J}}(\mathcal{V}, \mathcal{E}),\]
  where $\mathcal{E} \to \mathcal{J}$ is the left fibration for
  $\Psi$. We now have an obvious equivalence
  \[ \Map_{/\mathcal{J}}(\mathcal{V}, \mathcal{E}) \simeq
    \Map_{/\mathcal{V}}(\mathcal{V}, q^{*}\mathcal{E}),\]
  so our weighted limit is naturally equivalent to the space of sections
  of the left fibration $q^{*}\mathcal{E}$. Since pullback of left
  fibrations corresponds to composition of functors to $\mathcal{S}$,
  this is the left fibration for $\Psi \circ q$.  Moreover, the space
  of sections of a left fibration is equivalent to the limit of the
  corresponding functor to $\mathcal{S}$ by \cite{HTT}*{Corollary
    3.3.3.4}, so that we have
  \[ \Map_{/\mathcal{V}}(\mathcal{V}, q^{*}\mathcal{E}) \simeq
    \lim_{\mathcal{V}} \Psi \circ q,\] 
  as required.
\end{proof}

\begin{cor}\label{wtcolimrfib}
  Suppose $p \colon \mathcal{W} \to \mathcal{I}$ is the right
  fibration corresponding to a presheaf $W \colon \mathcal{I}^{\op}
  \to \mathcal{S}$. Then for a functor $\phi \colon \mathcal{I} \to
  \mathcal{C}$ there is an equivalence
  \begin{equation}
    \label{eq:wtcolimrfibeq}
    \colim^{W}_{\mathcal{I}} \phi \simeq \colim_{\mathcal{W}} \phi
    \circ p,
  \end{equation}
  provided either side exists.
\end{cor}
\begin{proof}
  By \cref{wtlimlfib} we have an equivalence
  \[ \lim^{W}_{\mathcal{I}^{\op}} \Map_{\mathcal{C}}(\phi, c) \simeq
    \lim_{\mathcal{W}^{\op}} \Map_{\mathcal{C}}(\phi p, c),\]
  natural in $c \in \mathcal{C}$. This implies \cref{eq:wtcolimrfibeq}
  by the universal mapping properties of the two objects.
\end{proof}

\begin{cor}
  The definition of weighted limit from \cite{RovelliWeight} agrees
  with the definition as an end.
\end{cor}
\begin{proof}
  Combine \cref{wtlimlfib} with
  \cite{RovelliWeight}*{Theorem D}.
\end{proof}

As a special case of these descriptions of weighted (co)limits we get
another characterization of (co)ends:
\begin{cor}\label{cor:coendweight}
  For a functor $F \colon \mathcal{I}^{\op} \times \mathcal{I} \to
  \mathcal{C}$ there are natural equivalences
  \[ \int_{\mathcal{I}} F \simeq
    \colim_{\mathcal{I} \times \mathcal{I}^{\op}}^{\Map_{\mathcal{I}}(\blank,\blank)} F, \qquad
    \int^{*}_{\mathcal{I}} F \simeq
    \lim_{\mathcal{I}^{\op} \times
      \mathcal{I}}^{\Map_{\mathcal{I}}(\blank,\blank)} F, \]
  provided either side exists in $\mathcal{C}$.
\end{cor}
\begin{proof}
  By \cite{HA}*{Proposition 5.2.1.11} the mapping space functor
  $\Map_{\mathcal{I}}(\blank,\blank)$ corresponds to the right
  fibration $\Tw^{r}(\mathcal{I}) \xto{\eta_{\mathcal{I}}^{\op}} \mathcal{I} \times
  \mathcal{I}^{\op}$ and hence the left fibration $\TwL(\mathcal{I})
  \xto{\eta_{\mathcal{I}}} \mathcal{I}^{\op}\times \mathcal{I}$.
  From \cref{wtcolimrfib} and \cref{wtlimlfib} we therefore have
  natural equivalences
  \[
    \colim_{\mathcal{I} \times
      \mathcal{I}^{\op}}^{\Map_{\mathcal{I}}(\blank,\blank)} F
    \simeq \colim_{\Tw^{r}(\mathcal{I})} F \circ
     \eta_{\mathcal{I}}^{\op}, \qquad \lim_{\mathcal{I}^{\op} \times
      \mathcal{I}}^{\Map_{\mathcal{I}}(\blank,\blank)} F
    \simeq \lim_{\TwL(\mathcal{I})} F \circ
     \eta_{\mathcal{I}}.\]
   The claims then follow by combining these equivalences with those
   of \cref{cor:endviatw} and \cref{rmk:coendtw}.
 \end{proof}

\begin{remark}
  Cordier and Porter~\cite{CordierPorter} introduced homotopy-coherent
  ends and coends in the context of simplicial categories using a
  version of the characterization in \cref{cor:coendweight}: they
  consider simplicially enriched (co)limits weighted by a cofibrant
  replacement of the mapping space functor in a simplicial category.
\end{remark}

We can also give a characterization of weighted colimits in
$\mathcal{S}$ via the universal property of presheaves as a free
cocompletion:
\begin{propn}\label{propn:colimWPSh}
  For $W \colon \mathcal{I} \to \mathcal{S}$, the unique
  colimit-preserving functor
  \[\Yo_{\mathcal{I},!}W\colon \mathcal{P}(\mathcal{I}) \to
    \mathcal{S}\] extending $W$ is equivalent to the weighted colimit functor
  \[\colim^{W}_{\mathcal{I}^{\op}} \colon
  \Fun(\mathcal{I}^{\op},\mathcal{S}) \to \mathcal{S}.\]
\end{propn}
\begin{proof}
  It suffices to show that $\colim^{W}_{\mathcal{I}^{\op}}$ preserves
  colimits, and that $\colim^{W}_{\mathcal{I}^{\op}} \circ
  \Yo_{\mathcal{I}}$ is equivalent to $W$. The former follows from the
  equivalence
  \[ \Map_{\mathcal{S}}(\colim^{W}_{\mathcal{I}^{\op}} \phi, X) \simeq
    \Map_{\Fun(\mathcal{I},\mathcal{S})}(W, \Map_{\mathcal{S}}(\phi,
    X))\]
  provided by \cref{thm:natend} (using that colimits in functor
  \icats{} are computed objectwise), while the latter follows from
  \cref{rmk:Yonedawt}.
\end{proof}

\begin{propn}
  Suppose $\mathcal{C}$ is a presentable \icat{}. Then for $W \colon
  \mathcal{I} \to \mathcal{S}$, the weighted
  colimit functor \[\colim^{W}_{\mathcal{I}^{\op}} \colon
    \Fun(\mathcal{I}^{\op}, \mathcal{C}) \to \mathcal{C}\]
  is equivalent to the tensor product
  \[ \Yo_{\mathcal{I},!}W \otimes \id_{\mathcal{C}} \colon
    \mathcal{P}(\mathcal{I}) \otimes \mathcal{C} \to
    \mathcal{S}\otimes \mathcal{C}\]
  in the \icat{} $\PrL$ of presentable \icats{}.
\end{propn}
\begin{proof}
  The equivalence
  $\mathcal{P}(\mathcal{I}) \otimes \mathcal{C} \simeq
  \Fun(\mathcal{I}^{\op}, \mathcal{C})$ from \cite{HA}*{Proposition
    4.8.1.17} is determined by a functor
  $\mathcal{P}(\mathcal{I}) \times \mathcal{C} \to \Fun(\mathcal{I}^{\op},
  \mathcal{C})$ that preserves colimits in each variable, by the
  universal property of the cocomplete tensor product. This
  corresponds to the functor
  $\mathcal{P}(\mathcal{I}) \times \mathcal{C} \times \mathcal{I}^{\op} \to
  \mathcal{C}$ given by
  \[ (\phi, c, i) \mapsto \phi(i) \otimes c,\]
  using the tensoring of $\mathcal{C}$ over $\mathcal{S}$ (given by
  taking colimits of constant functors).

The functor $\colim^{W}_{\mathcal{I}^{\op}}$ preserves colimits by the
same argument as in the proof of \cref{propn:colimWPSh} and is therefore
determined by the composite
\[ \mathcal{P}(\mathcal{I}) \times \mathcal{C} \to \Fun(\mathcal{I}^{\op},\mathcal{C}) \to \mathcal{C},
  \qquad (\phi, c) \mapsto \colim^{W}_{\mathcal{I}^{\op}} (\phi
  \otimes c).\]
Now using \cref{thm:natend} we have natural equivalences
\[
  \begin{split}
  \Map_{\mathcal{C}}(\colim^{W}_{\mathcal{I}^{\op}} (\phi
  \otimes c), x) & \simeq \Map_{\Fun(\mathcal{I},\mathcal{S})}(W,
  \Map_{\mathcal{C}}(\phi \otimes c, x)) \\
   & \simeq
  \Map_{\Fun(\mathcal{I},\mathcal{S})}(W, \Map_{\mathcal{S}}(\phi,
  \Map_{\mathcal{C}}(c, x))) \\
   & \simeq
  \Map_{\mathcal{S}}(\colim^{W}_{\mathcal{I}^{\op}} \phi,
  \Map_{\mathcal{C}}(c, x)) \\
   & \simeq \Map_{\mathcal{C}}((\colim^{W}_{\mathcal{I}^{\op}}\phi)
  \otimes c, x),
  \end{split}
\]
which gives a natural equivalence
\[ \colim^{W}_{\mathcal{I}^{\op}} (\phi
  \otimes c) \simeq (\colim^{W}_{\mathcal{I}^{\op}}\phi)
  \otimes c\]
using the Yoneda Lemma. The right-hand side here is the functor
  \[ \mathcal{P}(\mathcal{I}) \times \mathcal{C} \to \mathcal{C},
    \qquad (\phi, c) \mapsto (\Yo_{\mathcal{I},!}W)(\phi) \otimes c \]
by \cref{propn:colimWPSh}, or in other words the composite
\[ \mathcal{P}(\mathcal{I}) \times \mathcal{C}
  \xto{\Yo_{\mathcal{I},!}W \times \id_{\mathcal{C}}} \mathcal{S} \times
  \mathcal{C} \xto{\otimes} \mathcal{C},\]
which corresponds to $(\Yo_{\mathcal{I},!}W) \otimes
\id_{\mathcal{C}}$ composed with the equivalence $\mathcal{S} \otimes
\mathcal{C} \isoto \mathcal{C}$, as required.
\end{proof}

Combining \cref{cor:coendweight} with the last two results we get the
description of the coend functor that is taken as the definition in
\cite{AnelLejay}:
\begin{cor}\label{cor:coendpsh} 
  Let $\mathcal{I}$ be a small \icat{}.
  \begin{enumerate}[(i)]
  \item The coend functor
    \[ \int_{\mathcal{I}} \colon \mathcal{P}(\mathcal{I}^{\op} \times
      \mathcal{I}) \to \mathcal{S}\] is equivalent to the unique
    colimit-preserving functor
    $\Yo_{\mathcal{I}^{\op} \times
      \mathcal{I},!}\Map_{\mathcal{I}}(\blank,\blank)$ extending
    $\Map_{\mathcal{I}}(\blank,\blank) \colon \mathcal{I}^{\op} \times
    \mathcal{I} \to \mathcal{S}$.
  \item If $\mathcal{C}$ is a presentable \icat{}, then the coend
    functor
    \[ \int_{\mathcal{I}} \colon \Fun(\mathcal{I} \times
      \mathcal{I}^{\op}, \mathcal{C}) \to \mathcal{C}\]
    is equivalent to the tensor product
    \[ \Yo_{\mathcal{I}^{\op} \times
      \mathcal{I},!}\Map_{\mathcal{I}}(\blank,\blank) \otimes
    \id_{\mathcal{C}} \colon \mathcal{P}(\mathcal{I}^{\op}\times
    \mathcal{I}) \otimes \mathcal{C} \to \mathcal{S} \otimes
    \mathcal{C}\]
  in $\PrL$. \qed
  \end{enumerate}
\end{cor}


\begin{bibdiv}
  \begin{biblist}
    \bib{AnelLejay}{article}{
  author={Anel, Mathieu},
  author={Lejay, Damien},
  title={Exponentiable Higher Toposes},
  date={2018},
  eprint={arXiv:1802.10425},
}

\bib{BarwickMackey}{article}{
  author={Barwick, Clark},
  title={Spectral {M}ackey functors and equivariant algebraic $K$-theory ({I})},
  journal={Adv. Math.},
  volume={304},
  date={2017},
  pages={646--727},
  eprint={arXiv:1404.0108},
  year={2014},
}

\bib{CordierPorter}{article}{
  author={Cordier, Jean-Marc},
  author={Porter, Timothy},
  title={Homotopy coherent category theory},
  journal={Trans. Amer. Math. Soc.},
  volume={349},
  date={1997},
  number={1},
  pages={1--54},
}

\bib{freepres}{article}{
  author={Gepner, David},
  author={Haugseng, Rune},
  author={Nikolaus, Thomas},
  title={Lax colimits and free fibrations in $\infty $-categories},
  eprint={arXiv:1501.02161},
  journal={Doc. Math.},
  volume={22},
  date={2017},
  pages={1225--1266},
}

\bib{GlasmanTHHHodge}{article}{
  author={Glasman, Saul},
  title={A spectrum-level {H}odge filtration on topological {H}ochschild homology},
  journal={Selecta Math. (N.S.)},
  volume={22},
  date={2016},
  number={3},
  pages={1583--1612},
  eprint={arXiv:1408.3065},
}

\bib{cois}{article}{
  author={Haugseng, Rune},
  author={Melani, Valerio},
  author={Safronov, Pavel},
  title={Shifted coisotropic correspondences},
  date={2019},
  eprint={arXiv:1904.11312},
}

\bib{LoregianEnd}{article}{
  author={Loregian, Fosco},
  date={2019},
  title={Coend calculus},
  eprint={arXiv:1501.02503},
}

\bib{LoregianFub}{article}{
  author={Loregian, Fosco},
  date={2019},
  title={A {F}ubini rule for $\infty $-coends},
  eprint={arXiv:1902.06086},
}

\bib{HTT}{book}{
  author={Lurie, Jacob},
  title={Higher Topos Theory},
  series={Annals of Mathematics Studies},
  publisher={Princeton University Press},
  address={Princeton, NJ},
  date={2009},
  volume={170},
  note={Available from \url {http://math.ias.edu/~lurie/}},
}

\bib{HA}{book}{
  author={Lurie, Jacob},
  title={Higher Algebra},
  date={2017},
  note={Available at \url {http://math.ias.edu/~lurie/}.},
}

\bib{MacLaneWorking}{book}{
  author={Mac Lane, Saunders},
  title={Categories for the working mathematician},
  series={Graduate Texts in Mathematics},
  volume={5},
  edition={2},
  publisher={Springer-Verlag, New York},
  date={1998},
}

\bib{MazelGeeGroth}{article}{
  author={Mazel-Gee, Aaron},
  title={On the Grothendieck construction for $\infty $-categories},
  journal={J. Pure Appl. Algebra},
  volume={223},
  date={2019},
  number={11},
  pages={4602--4651},
  eprint={arXiv:1510.03525},
}

\bib{RovelliWeight}{article}{
  author={Rovelli, Martina},
  title={Weighted limits in an $(\infty ,1)$-category},
  date={2019},
  eprint={arXiv:1902.00805},
}

\bib{ShahThesis}{article}{
  author={Shah, Jay},
  title={Parametrized higher category theory and higher algebra: Exposé II -- Indexed homotopy limits and colimits},
  date={2018},
  eprint={arXiv:1809.05892},
}

\bib{YonedaExt}{article}{
  author={Yoneda, Nobuo},
  title={On Ext and exact sequences},
  journal={J. Fac. Sci. Univ. Tokyo Sect. I},
  volume={8},
  date={1960},
  pages={507--576 (1960)},
}
\end{biblist}
\end{bibdiv}
\end{document}